\documentclass[12pt]{amsproc}

\usepackage{geometry, enumerate, amsmath, amssymb, amsthm, amscd, mathrsfs, hyperref, graphicx, color, enumerate, float, verbatim, colonequals, url}

\numberwithin{equation}{section}
\newtheorem{thm}[equation]{Theorem}
\newtheorem{cor}[equation]{Corollary}

\newtheorem{lem}[equation]{Lemma}

\newtheorem{prop}[equation]{Proposition}

\theoremstyle{definition}\newtheorem{Rem}[equation]{Remark}
\theoremstyle{definition}\newtheorem{Def}[equation]{Definition}
\theoremstyle{definition}
\theoremstyle{definition}

\newcommand{\F}{\mathbb{F}}

\newcommand{\GL}{{\rm GL}}

\renewcommand{\le}{\leqslant}
\renewcommand{\ge}{\geqslant}

\newcommand{\mcB}{\mathcal{B}}
\newcommand{\mcC}{\mathcal{C}}
\newcommand{\mcD}{\mathcal{D}}

\newcommand{\mcM}{\mathcal{M}}

\newcommand{\mcS}{\mathcal{S}}

\newcommand{\msC}{\mathscr{C}}
\newcommand{\msJ}{\mathscr{J}}

\newcommand{\msS}{\mathscr{S}}

\newcommand{\msZ}{\mathscr{Z}}

\newcommand{\lcm}{\text{lcm}}

\newcommand*\xbar[1]{%
   \hbox{%
     \vbox{%
       \hrule height 0.5pt 
       \kern0.5ex
       \hbox{%
         \kern-0.1em
         \ensuremath{#1}%
         \kern-0.1em
       }%
     }%
   }%
}

\title{A new infinite family of $\sigma$-elementary rings}

\date{\today}

\author{Eric Swartz}
\address{Department of Mathematics, William \& Mary, P.O. Box 8795, Williamsburg, VA 23187-8795, USA}
\email{easwartz@wm.edu}

\author{Nicholas J. Werner}
\address{Department of Mathematics, Computer and Information Science, State University of New York College at Old Westbury, Old Westbury, NY 11560, USA}
\email{wernern@oldwestbury.edu}

\begin{document}

\begin{abstract}
A cover of an associative (not necessarily commutative nor unital) ring $R$ is a collection of proper subrings of $R$ whose set-theoretic union equals $R$. If such a cover exists, then the covering number $\sigma(R)$ of $R$ is the cardinality of a minimal cover, and a ring $R$ is called $\sigma$-elementary if $\sigma(R) < \sigma(R/I)$ for every nonzero two-sided ideal $I$ of $R$.  In this paper, we provide the first examples of $\sigma$-elementary rings $R$ that have nontrivial Jacobson radical $J$ with $R/J$ noncommutative, and we determine the covering numbers of these rings. 
\end{abstract}

\maketitle

\section{Introduction}\label{sect:intro}

Throughout this paper, all rings are associative, but do not necessarily contain a multiplicative identity.  A \textit{cover} of a ring $R$ is a collection of proper subrings of $R$ whose set-theoretic union is all of $R$, where $S \subseteq R$ is a subring if $S$ is a group under addition and closed under multiplication; that is, a subring need not contain a multiplicative identity.  A ring need not have a cover---consider, for example, a finite field---but, assuming one exists, the ring is said to be \textit{coverable}, and we define the \textit{covering number} $\sigma(R)$ to be the cardinality of a minimal cover.  If no such cover exists, we say that $\sigma(R) = \infty$.

Analogous problems have been studied extensively for groups (see \cite{Bhargava, Britnell1, Britnell2, BryceFedriSerena, Cohn, DetomiLucchini, Garonzi, GaronziKappeSwartz, Holmes, Maroti, Scorza, Swartz, Tomkinson}). In recent years, covering numbers of rings have begun to receive considerable attention (see \cite{CaiWerner, Cohen, Crestani, LucchiniMaroti, PeruginelliWerner, SwartzWerner, Werner}).  In the case of groups, it is a famous theorem of Tomkinson \cite{Tomkinson} that no group has covering number 7. While many other examples of positive integers that are not the covering number of any group are known (see \cite{GaronziKappeSwartz}), it is an open problem to determine whether or not there are infinitely many positive integers that do not occur as the covering number of a group. To date, no similar theorems have been proved for rings, although it is conjectured that there is no ring with covering number 13. The purpose of the present article is to study a class of finite rings that we believe will help resolve that conjecture, and will be useful in determining covering numbers of other rings.

As noted in \cite[Proof of Proposition 3.1]{SwartzWerner}, results by Neumann \cite[Lemma 4.1, 4.4]{Neumann} and Lewin \cite[Lemma 1]{Lewin} together show that, if a ring $R$ has finite covering number, then there exists a finite homomorphic image of $R$ with the same covering number; in other words, to determine which integers are covering numbers of rings, it suffices to consider finite rings.  Moreover, by \cite[Theorem 3.12]{SwartzWerner}, if $R$ is a unital ring such that its covering number $\sigma(R)$ is finite, then there exists a two-sided ideal $I$ of $R$ such that $R/I$ is finite, $R/I$ has prime characteristic, the Jacobson radical $J$ of $R/I$ satisfies $J^2 = \{0\}$, and $\sigma(R/I) = \sigma(R)$.


The reductions mentioned above hint at a strategy, which was originally employed for the study of covers of groups in \cite{DetomiLucchini}.  It is not difficult to see that, if $I$ is a two-sided ideal of $R$ and $R/I$ has a cover, then the inverse images of the subrings in the cover under the natural epimorphism form a cover for $R$; that is, $\sigma(R) \le \sigma(R/I)$.  It follows that the rings of interest for determining covering numbers are those rings $R$ whose covering number is strictly less than the covering number of all proper quotients.  (For this reason, we use the convention that a ring without a proper cover has infinite covering number.)

\begin{Def}\label{def:sigmaelementary}
A ring $R$ is said to be \textit{$\sigma$-elementary} if $\sigma(R) < \sigma(R/I)$ for every nonzero two-sided ideal $I$ of $R$. Note that a $\sigma$-elementary ring $R$ must be coverable, since $\sigma(R) < \sigma(\{0\}) = \infty$.
\end{Def}

Evidently, determining which integers are covering numbers of rings amounts to classifying all $\sigma$-elementary rings and finding their covering numbers.  In some sense, there are four cases for $\sigma$-elementary rings $R$ with Jacobson radical $J$: (1) $R$ is semisimple (i.e., $J = \{0\}$) and $R/J$ is commutative, (2) $R$ is not semisimple and $R/J$ is commutative, (3) $R$ is semisimple and $R/J$ is not commutative, or (4) $R$ is not semisimple and $R/J$ is not commutative. 
Examples have been found for each of the first three cases (see \cite{Werner, SwartzWerner, LucchiniMaroti}), but, up to this point, it is an open question whether examples exist in the fourth case.

The main result of this paper is the identification of an infinite family of $\sigma$-elementary rings with nontrivial Jacobson radical and noncommutative semisimple quotient.  In order to define these rings, we must first introduce some notation.  We will assume that the reader is familiar with the basic theory of finite rings (e.g., the Artin-Wedderburn Theorem and the structure of finite semisimple rings) and their modules and bimodules. Given a prime power $q$, let $\F_q$ be the finite field with $q$ elements, and let $M_n(q)$ be the ring of $n \times n$ matrices with entries from $\F_q$. For a ring $R$, the Jacobson radical of $R$ is $\msJ(R)$, and if $R$ is unital, then $R^\times$ denotes the unit group of $R$. 

Next, let $q_1$ and $q_2$ be powers of $p$. We define $q_1 \otimes q_2$ to be the order of the field compositum of $\F_{q_1}$ and $\F_{q_2}$, which is $\F_{q_1} \otimes_{\F_p} \F_{q_2}$. Note that if $q_1 = p^{d_1}$ and $q_2 = p^{d_2}$, then $q_1 \otimes q_2 = p^{\lcm(d_1, d_2)}$.  

\begin{Def}\label{def:AGL}
Let $n \ge 1$, let $q_1$ and $q_2$ be powers of the same prime $p$, and let $q = q_1 \otimes q_2$. We define $A(n,q_1,q_2)$ to be the following subring of $M_{n+1}(q)$:
\begin{equation*}
A(n,q_1,q_2) := \left(\text{\begin{tabular}{c|c} $M_n(q_1)$ & $M_{n \times 1}(q)$ \\ \hline $0$ & $\F_{q_2}$ \end{tabular} }\right).
\end{equation*}
\end{Def}

The rings $A(n,q_1,q_2)$ were inspired by representations of the affine general linear group \rm{AGL}$(n,q)$, which is isomorphic to a subgroup of the unit group of $A(n,q,q)$. For this reason, we refer to $A(n,q_1,q_2)$ as a ring of \textit{AGL-type}. Note that by construction, $A(n,q_1,q_2)$ is noncommutative and has a nonzero radical.

Next, let $\omega$ denote the prime omega function, which counts the number of distinct prime divisors of a natural number. Note that $\omega(1) = 0$. Finally, let $\binom{n}{k}_q$ denote the $q$-binomial coefficient, which counts the number of $k$-dimensional subspaces of $\F_q^n$:
\begin{equation*}
\binom{n}{k}_q = \frac{(q^n - 1)(q^{n-1}-1) \cdots (q^{n - (k-1)}-1)}{(q^k -1)(q^{k-1} -1)\cdots (q-1)}.
\end{equation*}

We can now state our main result.

\begin{thm}\label{thm:main}
Let $R \cong A(n,q_1, q_2)$, where $n \ge 1$, and let $q := q_1 \otimes q_2 = q_1^d$. If $n \ge 2$, then let $a$ be the smallest prime divisor of $n$.  Then, $R$ is $\sigma$-elementary if and only if one of the following holds:
\begin{enumerate}
\item $n=1$ and $(q_1, q_2) \ne (2, 2)$ or $(4,4)$. In this case, $\sigma(R) = q+1$.
\item $n \ge 3$, $d < n - (n/a)$, and $(n, q_1) \ne (3, 2)$. In this case,
\begin{equation*}
\sigma(R) = q^n + \binom{n}{d}_{q_1} + \omega(d).
\end{equation*}
\end{enumerate}
\end{thm}

This paper is organized as follows.  In Section \ref{sect:prelim}, we collect a number of background results that will be used in the rest of the article.  Section \ref{sect:small} is devoted to an analysis of the rings $A(n, q_1, q_2)$ when $n = 1,2$, Section \ref{sect:setup} sets up the machinery for case when $n \ge 3$, and Section \ref{sect:AGL} is dedicated to the proof of Theorem \ref{thm:main}.


\section{Frequently used results}\label{sect:prelim}
Here, we collect some previously established results that will be used often in later sections. Each part of our first lemma is straightforward to prove, and we will apply these properties without attribution going forward.

\begin{lem}\label{lem:basics}
Let $R$ be a ring with unity.
\begin{enumerate}[(1)]
\item $R$ is coverable if and only if $R$ cannot be generated (as a ring) by a single element.
\item If $R$ is noncommutative, then $R$ is coverable.
\item For any two-sided ideal $I$ of $R$, a cover of $R/I$ can be lifted to a cover of $R$. Hence, $\sigma(R) \le \sigma(R/I)$.
\item If each proper subring of $R$ is contained in a maximal subring, then we may assume that any minimal cover of $R$ consists of maximal subrings.
\end{enumerate}
\end{lem}

Slightly more substantial is the next lemma, which can ensure that a maximal subring of $R$ is included in every cover of $R$.

\begin{lem}\label{lem:SigmaElementary}
Let $R$ be a coverable ring, and let $\mcC$ be a minimal cover of $R$. 
\begin{enumerate}[(1)]
\item \cite[Lemma 2.1]{Werner} If $M$ is a maximal subring of $R$ and $M \notin \mcC$, then $\sigma(M) \leq \sigma(R)$.

\item \cite[Lemma 2.2]{SwartzWerner} Let $S$ be a proper subring of $R$ such that $R=S \oplus I$ for some two-sided ideal $I$ of $R$. If $\sigma(R) < \sigma(R/I)$, then $S \subseteq T$ for some $T \in \mcC$. If, in addition, $S$ is a maximal subring of $R$, then $S \in \mcC$.

\end{enumerate}
\end{lem}

Any ring of AGL-type is a finite ring of characteristic $p$. The Wedderburn-Malcev Theorem \cite[Sec.\ 11.6, Cor.\ p.\ 211]{Pierce}, \cite[Thm.\ VIII.28]{McDonald} (also known as the Wedderburn Principal Theorem) describes the basic structure of such rings.

\begin{thm}\label{thm:Wedderburn} (Wedderburn-Malcev Theorem)
Let $R$ be a finite ring with unity of characteristic $p$. Then, there exists an $\F_p$-subalgebra $S$ of $R$ such that $R = S \oplus \msJ(R)$, and $S \cong R/\msJ(R)$ as $\F_p$-algebras. Moreover, $S$ is unique up to conjugation by elements of $1+\msJ(R)$.
\end{thm}

By Theorem \ref{thm:Wedderburn}, all semisimple complements to $\msJ(R)$ in $R$ are conjugate. We let $\msS(R)$ be the set of all such complements. Thus, for any $S \in \msS(R)$, we have $R = S \oplus \msJ(R)$, and $S \cong R/\msJ(R)$.

Aside from providing a valuable decomposition for rings of characteristic $p$, the Wedderburn-Malcev Theorem affords a way to classify maximal subrings of $R$.

\begin{lem}\label{lem:MaxSubringClassification}\cite[Theorem 3.10]{SwartzWerner}
Let $R$ be a finite ring with unity of characteristic $p$, let $M$ be a maximal subring of $R$, and let $J = \msJ(R)$.
\begin{enumerate}[(1)]
\item $J \subseteq M$ if and only if $M$ is the inverse image of a maximal subring of $R/J$.
\item $J \not\subseteq M$ if and only if $M = S \oplus \msJ(M)$, where $S \in \msS(R)$ and $\msJ(M) = M \cap J$ is an ideal of $R$ that is maximal among the subideals of $R$ contained in $J$.
\end{enumerate}
\end{lem}


\section{Basic properties and small cases}\label{sect:small}
Recall from the introduction that when $q_1$ and $q_2$ are powers of the same prime $p$, we define $q_1 \otimes q_2$ to be the order of $\F_{q_1} \otimes_{\F_p} \F_{q_2}$. For $n \ge 1$, we let $A(n,q_1,q_2)$ be the following subring of $M_{n+1}(q)$:
\begin{equation*}
A(n,q_1,q_2) := \left(\text{\begin{tabular}{c|c} $M_n(q_1)$ & $M_{n \times 1}(q)$ \\ \hline $0$ & $\F_{q_2}$ \end{tabular} }\right),
\end{equation*}
where $q = q_1 \otimes q_2$. The definition of these rings is motivated by the matrix representation of the affine general linear group $\textnormal{AGL}(n,q)$, which is isomorphic to the subgroup
\begin{equation*}
\left(\text{\begin{tabular}{c|c} $\textnormal{GL}(n,q)$ & $M_{n \times 1}(q)$ \\ \hline $0$ & $1$ \end{tabular} }\right)
\end{equation*}
of the unit group of $A(n,q,q)$.

The goal of this section is describe basic properties of $A(n,q_1,q_2)$ and, for $n \le 2$, to calculate its covering number and determine when it is $\sigma$-elementary (Theorem \ref{thm:smalln}). The corresponding results for $n \ge 3$ are given later in Corollary \ref{cor:AGLn>1} and Theorem \ref{thm:AGLupperbound}. Our strategy is based on the approach used to prove \cite[Theorem 1.7]{GaronziKappeSwartz}, where it is shown that for $n \ne 2$, the group $\textnormal{AGL}(n,q)$ has covering number $(q^{n+1}-1)/(q-1)$. However, our formulas for $\sigma(A(n,q_1,q_2))$---and the methods used to establish them---are more complicated than in the group case.

Throughout this section, $R=A(n,q_1,q_2)$, and $S_1$ and $S_2$ are the following subrings of $R$:
\begin{equation*}
S_1 := \left(\text{\begin{tabular}{c|c} $M_n(q_1)$ & $0$ \\ \hline $0$ & $0$ \end{tabular} }\right) \cong M_n(q_1), \quad \text{ and } \quad S_2 := \left(\text{\begin{tabular}{c|c} $0$ & $0$ \\ \hline $0$ & $\F_{q_2}$ \end{tabular} }\right) \cong \F_{q_2}.
\end{equation*}
The Jacobson radical of $R$ is
\begin{equation*}
J:= \left(\text{\begin{tabular}{c|c} $0$ & $M_{n \times 1}(q)$ \\ \hline $0$ & $0$ \end{tabular} }\right),
\end{equation*}
which is a simple $(S_1, S_2)$-bimodule of order $q^n$. We take $S:=S_1 \oplus S_2$ to be a fixed semisimple complement to $J$ in $R$. Note that $S \cong R/J$. 

The set of all semisimple complements to $J$ in $R$ is denoted by $\msS(R)$. By Theorem \ref{thm:Wedderburn}, $S$ is unique up to conjugation by elements of $1+J$, so $\msS(R) = \{S^{1+x} : x \in J \}$. It is clear that $J^2=\{0\}$, so $(1+x)^{-1} = 1-x$ for all $x \in J$. Hence, $S^{1+x} = \{ s + sx-xs : x \in J\}$ for each $x \in J$.

\begin{lem}\label{lem:Jcenter}
$J \cap Z(R) = \{0\}$, and consequently $|\msS(R)| = |J|$.
\end{lem}
\begin{proof}
Let $e_1$ and $e_2$ be the unit elements of $S_1$ and $S_2$, respectively. Then, $e_1 x = x = x e_2$ for any $x \in J$. However, $e_1$ and $e_2$ annihilate one another, so if $x \in J \cap Z(R)$, then 
\begin{equation*}
x = e_1 x = x e_1 = (x e_2) e_1 = 0.
\end{equation*}
Now, for all $s \in S$ and $x \in J$, we have $s^{1+x} = s + sx-xs$. So, having $S^{1+x} = S$ is equivalent to having $sx-xs \in S$. But, $sx - xs \in J$, and $S \cap J = \{0\}$, so $S^{1+x}=S$ if and only if $x \in J \cap Z(R)$. Thus, $|\msS(R)| = |J:(J \cap Z(R))| = |J|$.
\end{proof}

We wish to determine the covering number of $R$, and to classify exactly when $R$ is a $\sigma$-elementary ring. It is easy to check that the latter condition holds exactly when $\sigma(R) < \sigma(S)$.

\begin{lem}\label{AGL:sigmaelem}
The ring $R$ is $\sigma$-elementary if and only if $\sigma(R) < \sigma(S)$.
\end{lem}
\begin{proof}
Because $S_1$ and $S_2$ are simple rings and $J^2 = \{0\}$, the only nonzero, proper two-sided ideals of $R$ are $J$, $S_1 \oplus J$, and $S_2 \oplus J$. The corresponding residue rings of $R$ are isomorphic to $S$, $S_2$, and $S_1$, respectively. The result follows after noting that $\sigma(S) \le \min\{\sigma(S_1), \sigma(S_2)\}$.
\end{proof}

Next, we give a proposition that fundamentally determines the covering number of $R$. For $x \in J$, let $C_S(x) := \{s \in S : sx=xs\}$ be the centralizer of $x$ in $S$.

\begin{prop}\label{prop:AGLcover}
Let $\mcM$ be a minimal set of maximal subrings of $S$ such that
\begin{equation*}
\bigcup_{x \in J\backslash \{0\}} C_S(x) \subseteq \bigcup_{M \in \mcM} M,
\end{equation*}
and define 
\begin{equation*}
\msZ := \{ M \oplus J : M \in \mcM\}.
\end{equation*}
Then,
\begin{enumerate}[(1)]
\item $\msS(R) \cup \msZ$ is a cover of $R$.
\item If $R$ is $\sigma$-elementary, then $\msS(R) \cup \msZ$ is a minimal cover of $R$.
\item $|\msZ| \le (|J|-1)/(q-1) = (q^{n}-1)/(q-1)$.
\item If $R$ is $\sigma$-elementary, then $q^n + 1 \le \sigma(R) \le (q^{n+1}-1)/(q-1)$.
\end{enumerate}
\end{prop}

\begin{proof}
(1) We first show that adding in the elements of $\msZ$ to those of $\msS(R)$ constitutes a cover. We know that $\msS(R) = \{S^{1+x}: x \in J\}$,
and so \[ \bigcup_{S_0 \in \msS(R)} S_0 = S^{1+J}.\]
Moreover, since $s^{1 + x} = s + (sx - xs) \subseteq s + J$, it suffices to cover the elements $s$ for which $s^{1 + J} \subsetneq s + J$ by maximal subrings of $S$.  Note that since $0_S, 1_S \in C_S(x)$ for any $x \in J$, $S^{1 + J} \subsetneq R$, so the maximal subrings in $\msS(R)$ alone do not constitute a cover.
 
Suppose $s^{1 + J} \subsetneq s + J$ for a fixed $s \in S$.  Since $|s + J| = |J|$, this implies that $s^{1 + x} = s^{1 + y}$ for some $x \neq y \in J$.  However, this means $s^{1 + (x - y)} = s$ and $s \in C_S(x-y)$.  Conversely, if $s \notin \bigcup_{x \in J\backslash \{0\}} C_S(x)$, then $|s^{1+J}| = |J| = |s + J|$, which implies that $s + J = s^{1 + J} \subseteq S^{1 + J}$.  By Lemma \ref{lem:MaxSubringClassification}, every maximal subring of $R$ not contained in $\msS(R)$ is of the form $T \oplus J$, where $T$ is a maximal subring of $S$.  Thus, extending $\msS(R)$ to a cover of $R$ amounts precisely to finding a set $\mcM$ of maximal subrings of $S$ such that
 \[ \bigcup_{x \in J\backslash \{0\}} C_S(x) \subseteq \bigcup_{M \in \mcM} M.\]  That $\msS(R) \cup \msZ$ is a cover follows.
 
(2) Assume that $R$ is $\sigma$-elementary. Then, $\sigma(R) < \sigma(S)$.  By Lemma \ref{lem:SigmaElementary} (2), every complement to $J$ must be contained in any minimal cover of $R$.  Thus, by the arguments above, extending $\msS(R)$ to a minimal cover of $R$ amounts to finding a minimal cover of $\bigcup_{x \in J\backslash \{0\}} C_S(x)$.

(3) Note that $C_S(cx) = C_S(x)$ whenever $c \in \F_q^\times$. This means that there are at most $(|J|-1)/(q-1)$ subrings $C_S(x)$ of $S$. Hence, $|\msZ| \le (|J|-1)/(q-1)$. 

(4) By Lemma \ref{lem:Jcenter}, $|\msS(R)| = |J| = q^n$. As in part (2), each subring in $\msS(R)$ is contained in every minimal cover of $R$ because $R$ is $\sigma$-elementary. But, every conjugate of $S$ contains $0$, so the union of all the subrings in $\msS(R)$ has cardinality at most
\begin{equation*}
|\msS(R)|(|S|-1)+1 = |J|(|S|-1) + 1< |R|.
\end{equation*}
Thus, using part (1) and part (3),
\begin{equation*}
|\msS(R)|+1 = q^n+1 \le \sigma(R) \le q^n + \frac{q^n-1}{q-1} = \frac{q^{n+1}-1}{q-1},
\end{equation*}
which completes the proof.
\end{proof}

When $n =1$ or $n=2$, we can determine the covering number of $R$ and describe when $R$ is $\sigma$-elementary by comparing the bound for $\sigma(R)$ in Proposition \ref{prop:AGLcover}(4) with the value of $\sigma(S)$. For larger values of $n$, we will construct a cover of the union of all the centralizers $C_S(x)$ (see Proposition \ref{prop:AGLn>1cover}), which will allow us to find $\sigma(R)$.

Since $R/(S_2 \oplus J) \cong S_1 \cong M_n(q_1)$, $\sigma(M_n(q_1))$ provides an upper bound for $\sigma(R)$. When $n \ge 2$, $M_n(q_1)$ is noncommutative, and hence is coverable (in fact, being a simple ring, $M_n(q_1)$ is $\sigma$-elementary), and $\sigma(M_n(q_1))$ was determined in \cite{LucchiniMarotiARXIV, Crestani}. The formula for $\sigma(M_n(q_1))$ is somewhat complicated and will be cited below. For our purposes it will be enough to have an upper bound on this covering number.

\begin{lem}\label{lem:bounds}
Let $n \ge 2$, let $a$ be the smallest prime divisor of $n$, and let $d \ge n - (n/a)$. Then, for every prime power $q$, $\sigma(M_n(q)) \le q^{nd}$, with equality only when $n=q=2$ and $d=1$.
\end{lem}
\begin{proof}
By \cite[Section 7]{LucchiniMarotiARXIV} and \cite[Theorem A]{Crestani},
\begin{equation*}
\sigma(M_n (q)) = \frac{1}{a} \prod_{k=1,\\ a \nmid k}^{n-1} (q^n - q^k) + \sum_{k=1,\\ a \nmid k}^{\lfloor n/2 \rfloor} {n \choose k}_q.
\end{equation*}
It suffices to prove the lemma for $d = n - n/a$. When $2 \le n \le 4$, the inequality can be verified by inspection using the above formula. Note that when $n=2$, we have $\sigma(M_2(q)) = \tfrac{1}{2}(q^2+q+2)$, which is less than or equal to $q^2$ for all $q \ge 2$, and is exactly $q^2$ when $q=2$. So, assume that $n \ge 5$.

Now, since $d$ is the number of factors in $\prod_{k=1,\\ a \nmid k}^{n-1} (q^n - q^k)$, we certainly have 
\begin{equation*}
\frac{1}{a} \prod_{k=1,\\ a \nmid k}^{n-1} (q^n - q^k) < \frac{1}{2}(q^n)^d.
\end{equation*}
So, it suffices to bound $\sum_{k=1,\\ a \nmid k}^{\lfloor n/2 \rfloor} {n \choose k}_q$ by $\tfrac{1}{2}q^{nd}$. Note that since $a \ge 2$, $nd \ge (n^2)/2$. We will show that the aforementioned sum is bounded above by $\tfrac{1}{2}q^{(n^2)/2}$.

Assume first that $q > 2$. One form of the $q$-Binomial Theorem (also known as Cauchy's Binomial Theorem) states that
\begin{equation*}
\prod_{k=0}^{n-1} (1+q^k) = \sum_{k=0}^n q^{k(k-1)/2}\binom{n}{k}_q.
\end{equation*}
Using this and recalling that $n \ge 5$,
\begin{equation*}
\sum_{k=1,\\ a \nmid k}^{\lfloor n/2 \rfloor} {n \choose k}_q < \frac{1}{2}\sum_{k=0}^n q^{k(k-1)/2} \binom{n}{k}_q = \frac{1}{2}\prod_{k=0}^{n-1} (1+q^k).
\end{equation*}

Let $P = \prod_{k=0}^{n-1} (1+q^k)$. For all $1 \le k \le \lfloor n/2 \rfloor -1$, we see that
\begin{equation}\label{q^n ineq}
(1+q^k)(1+q^{n-1-k}) \le \sum_{k=0}^{n-1} q^k < q^n.
\end{equation}
Also, when $k=0$, $(1+q^k)(1+q^{n-1}) = 2(1+q^{n-1})$, which is less than $q^n$ because $q \ne 2$. Thus, if $n$ is even then $P$ consists of $n/2$ product pairs $(1+q^k)(1+q^{n-1-k})$, each of which is less than $q^n$. If $n$ is odd, then there are $\lfloor n/2 \rfloor - 1$ such pairs, along with one factor $1+q^{\lfloor n/2 \rfloor}$, which is less than $q^{n/2}$ because $n \ge 5$. Hence, $P < (q^{n/2})^n$, and we are done.

In the case where $q=2$, \eqref{q^n ineq} still holds, but $2(1+q^{n-1}) \not\le q^n$. To correct for this, we group together the first two and last two factors of $P$ and note that
\begin{align*}
(1+1)(1+2)(1+2^{n-2})(1+2^{n-1}) &= 6(1+2^{n-2}+2^{n-1}+2^{2n-3})\\
&\le 6(1 + \tfrac{1}{16}\cdot 2^{2n-3} + \tfrac{1}{8}\cdot 2^{2n-3} + 2^{2n-3})\\
&< 6(\tfrac{5}{4} \cdot 2^{2n-3}) < 2^{2n}.
\end{align*}
The argument then proceeds as before, and we obtain the desired inequality.
\end{proof}

\begin{thm}\label{thm:smalln}
Let $d$ be such that $q=q_1^d$. When $n \ge 2$, let $a$ be the smallest prime divisor of $n$.
\begin{enumerate}
\item If $n=1$, then $R$ is $\sigma$-elementary if and only if $(q_1, q_2) \ne (2,2)$ or $(4,4)$. If $(q_1,q_2)=(2,2)$, then $\sigma(R)=3$. If $(q_1,q_2)=(4,4)$, then $\sigma(R)=4$. Otherwise, $\sigma(R)=q+1$.
\item If $n \ge 2$ and $d \ge n-(n/a)$, then $R$ is not $\sigma$-elementary, and $\sigma(R) = \sigma(M_n(q_1))$. In particular, $R$ is not $\sigma$-elementary when $n=2$.
\end{enumerate}
\end{thm}
\begin{proof}
(1) When $n=1$, $S \cong \F_{q_1} \oplus \F_{q_2}$ is coverable if and only if $(q_1,q_2)=(2,2)$ or $(4,4)$ \cite[Corollary 3.8]{Werner}. In the former case, $\sigma(R) = \sigma(\F_2 \oplus \F_2)=3$, and in the latter case $\sigma(R) = \sigma(\F_4 \oplus \F_4) = 4$ \cite[Theorem 1.2]{CaiWerner}. In both instances, $R$ is not $\sigma$-elementary. On the other hand, if $S$ is not coverable, then neither $S_1$ nor $S_2$ is coverable. However, $R$ is coverable (since it is a noncommutative ring), and hence is $\sigma$-elementary. By Proposition \ref{prop:AGLcover}(4), $\sigma(R)=q+1$.

(2) Let $n \ge 2$, let $d \ge n-(n/a)$, and suppose that $R$ is $\sigma$-elementary. Then, $q^n = q_1^{nd}$, and by Proposition \ref{prop:AGLcover}(4) and Lemma \ref{lem:bounds},
\begin{equation*}
q_1^{nd}+1 \le \sigma(R) \le \sigma(M_n(q_1)) \le q_1^{nd},
\end{equation*}
a contradiction. So, $R$ cannot be $\sigma$-elementary. Thus, $\sigma(R) = \min\{\sigma(S), \sigma(S_1), \sigma(S_2)\}$. Since $S_2$ is not coverable and $\sigma(S)=\sigma(S_1)$ by \cite[Proposition 5.4]{PeruginelliWerner}, we get $\sigma(R) = \sigma(S_1) = \sigma(M_n(q_1))$. Finally, when $n=2$, $n-(n/a)=1$, so $d \ge n-(n/a)$ and $R$ is not $\sigma$-elementary.
\end{proof}

\begin{Rem}\label{rem:unital1}
In Theorem \ref{thm:smalln} (1) when $R = A(1, q_1, q_2)$ is $\sigma$-elementary, there exists a minimal cover of $R$ as constructed in Proposition \ref{prop:AGLcover} of the form $\msS(R) \cup \{M \oplus J\}$, where $M$ here is the set of all scalar matrices; that is, the cover consists of the complements to $J$ and a single additional subring. Each subring in this cover contains $1_R$. Thus, while it is our convention that subrings of $R$ need not contain $1_R$, this minimal cover of $\sigma(R)$ consists of unital subrings of $R$. 
\end{Rem}

\section{Setup for the general case}\label{sect:setup}

Maintaining the notation used in Section \ref{sect:small}, $R = A(n,q_1,q_2)$, where $n \ge 1$ and $q=q_1 \otimes q_2 = q_1^d$. Theorem \ref{thm:smalln} handles the situations where $n \le 2$, and also any case with $d \ge n-(n/a)$, where $n \ge 2$ and $a$ is the smallest prime divisor of $n$. Consequently, for the remainder of the paper, we will assume that $n \ge 3$ and $d < n-(n/a)$. 

Following Proposition \ref{prop:AGLcover}, we seek a minimal cover by maximal subrings of $S$ of the union of the centralizers $C_S(x)$. By \cite[Theorem 4.5]{PeruginelliWerner}, any maximal subring of $S$ has the form $M_1 \oplus S_2$ or $S_1 \oplus M_2$, where $M_i$ is a maximal subring of $S_i$. If $q_2 > p$, then the maximal subrings of $S_2 \cong \F_{q_2}$ are the maximal subfields of $\F_{q_2}$; otherwise, $\{0\}$ is the only maximal subring of $S_2$. The maximal subrings of $S_1 \cong M_n(q_1)$ were classified in \cite[Theorem 3.3]{PeruginelliWerner}. They fall into three types, which we summarize in the next definition. 

\begin{Def}\label{def:maxsubs}
Let $q$ be a prime power, and $M$ a maximal subring of $M_n(q)$.
\begin{itemize}
\item We call $M$ \textit{Type I-k} if $M$ is the stabilizer of a proper nonzero subspace of $\F_q^n$ of dimension $k$.  That is, for a fixed subspace $W$ of of $\F_q^n$ of dimension $k$,
 \[ M = \{A \in M_n(q) : AW \subseteq W\}.\]
Collectively (or if the value of $k$ is unimportant), we refer to these subrings simply as Type I.

\item We call $M$ \textit{Type II-$\ell$} if $M$ is the centralizer of a minimal degree field extension of $\F_q$ in $M_n(q)$.  That is, for $\ell$ a prime dividing $n$ and $K \subseteq M_n(q)$, $K \cong \F_{q^\ell}$, we have
 \[ M = C_{M_n(q)}(K) = \{A \in M_n(q) : Au = uA \text{ for all } u \in K\}.\]  Collectively (or if the value of $\ell$ is unimportant), we refer to these subrings simply as Type II. Any Type II-$\ell$ maximal subring is isomorphic to $M_{n/\ell}(q^\ell)$.

\item We call $M$ \textit{Type III-r} if $M$ is a $\GL(n,q)$-conjugate of $M_n(r)$, where $\F_r$ is a (nonzero) maximal subfield of $\F_q$.  Collectively (or if there is no need to specify $r$), these subrings are called Type III. Note that Type III maximal subrings exist only when $q > p$.
\end{itemize}
\end{Def}

We now define a specific collection of maximal subrings of $S$, which we will prove is a cover of all the centralizers $C_S(x)$.

\begin{Def}\label{def:C}
Define $\msC$ be the set of maximal subrings of $S$ consisting of:
 \begin{itemize}
  \item all maximal subrings of the form $T_U \oplus \F_{q_2}$, where $U$ is a $d$-dimensional subspace of $\F_{q_1}^n$ and $T_U$ is the set of all elements of $M_n(q_1)$ stabilizing $U$, and
  \item all maximal subrings of the form $M_n(q_1) \oplus \F_{r}$, where $\F_{r}$ is a maximal subfield of $\F_{q_2}$ containing $\F_{q_1} \cap \F_{q_2}$.
 \end{itemize}
\end{Def}

\begin{Rem}\label{rem:unital}
Note that each subring in $\msC$ contains the multiplicative identity $1_R$ of $R$. As $\msC$ will play the role of the collection $\mcM$ in Proposition \ref{prop:AGLcover}, the minimal cover we will construct for $R$ is $\msS(R) \cup \msZ$, where $\msZ = \{M \oplus J : M \in \msC\}$, and each subring in this cover contains $1_R$. Thus, while it is our convention that subrings of $R$ need not contain $1_R$, the minimal cover we will use to calculate $\sigma(R)$ consists of unital subrings of $R$.
\end{Rem}

Recall that $\omega(d)$ is equal to the number of distinct prime divisors of a natural number $d$. Since $d \ge 1$ is such that $q = q_1 \otimes q_2 = q_1^d$, we see that $\omega(d)$ counts the number of maximal subfields of $\F_{q_2}$ containing $\F_{q_1} \cap \F_{q_2}$ (which is the same as the number of maximal subfields of $\F_q$ containing $\F_{q_1}$). If $d=1$, then $q=q_1$ and $\F_{q_2} \subseteq \F_{q_1}$. In this case, no proper subfield of $\F_{q_2}$ contains $\F_{q_1} \cap \F_{q_2}$, and $\omega(1)=0$.

\begin{prop}\label{prop:AGLn>1cover}
Let $n \ge 3$ and let $q = q_1 \otimes q_2 = q_1^d$.  Then $\msC$ is a cover of $\bigcup_{x \in J \backslash\{0\}} C_S(x)$ that has size ${n \choose d}_{q_1} + \omega(d)$.
\end{prop}

\begin{proof}
Note that $\msC$ has the specified size, because the number of maximal subrings of $S$ of the form $T_U \oplus \F_{q_2}$ is $\binom{n}{d}_{q_1}$, and the number of maximal subrings of the form $M_n(q_1) \oplus \F_{r}$ is equal to the number of such subfields $\F_{r}$, which is $\omega(d)$.

To see that $\msC$ is a cover, suppose $s \in C_S(v)$ for some $v \neq 0$ in $J$, and write $s = h + \beta$, where $h \in M_n(q_1)$ and $\beta \in \F_{q_2}$. Note that $\beta v = 0 = vh$.  If $\beta \in \F_{r}$ for a maximal subfield of $\F_{q_2}$ containing $\F_{q_1} \cap \F_{q_2}$, then $s$ is in a maximal subring of the form $M_n(q_1) \oplus \F_{r}$. Otherwise, $\F_q = \F_{q_1}[\beta]$, and so $\beta$ is the root of a monic irreducible polynomial $m(x)$ of degree $d$ in $\F_{q_1}[x]$. Since $s \in C_S(v)$,
\begin{equation*}
hv = sv =vs =v \beta.
\end{equation*}
Viewing $h$ as a linear transformation on $\F_{q}^n$ and $v$ as a vector in $\F_{q}^n$, we see that $\beta$ is an eigenvalue of $h$. Hence, $\beta$ is a root of the minimal polynomial $\mu(x)$ of $h$. However, $h$ has entries in $\F_{q_1}$, so $\mu(x) \in \F_{q_1}[x]$. Since $m(x)$ is irreducible over $\F_{q_1}$, this means that $m(x)$ divides $\mu(x)$. It follows that there is a $d$-dimensional subspace $U$ of $\F_{q_1}^n$ that is invariant under $h$. In other words, $h \in T_U$, and so $s \in T_U \oplus \F_{q_2}$. The result follows.
\end{proof}

Next, we must prove that $\msC$ is a minimal cover of $\bigcup_{x \in J \backslash\{0\}} C_S(x)$.  To do this, we will prove that $\msC$ is a minimal cover of a proper subset of $\bigcup_{x \in J \backslash\{0\}} C_S(x)$.  We now present a lemma that allows us to establish this minimality.  This lemma is a direct modification of \cite[Lemma 5.1]{GaronziKappeSwartz} (which itself was adapted from \cite[Lemma 3.1]{Swartz}). 

\begin{lem}\label{lem:mincovercriterion}
Let $\Pi$ be a set of elements of $R$, and suppose the following hold.
\begin{enumerate}[(i)]
 \item $I \subseteq I_R$, where $I_R$ is an index set for collections of maximal subrings of $R$, such that each $\mcM_i$ for $i \in I_R$ is a set of maximal subrings of $R$ and $\bigcup_{i \in I_R} \mcM_i$ is a partition of the set of all maximal subrings of $R$.
 \item $\mcD := \bigcup_{i\in I} \mcM_i$ is a cover of $\Pi$.
 \item The elements of $\Pi$ are partitioned among the subrings contained in $\mcD$, each subring in $\mcD$ contains elements of $\Pi$, and $\Pi_i$ is defined to be the set of elements covered by the subrings in $\mcM_i$.
 \item For any $i \in I$ and $j \in I_R$, if $M, M' \in \mcM_j$, then $|M \cap \Pi_i| = |M' \cap \Pi_i|$.
\end{enumerate}
For a maximal subring $M \notin \mcD$, define \[c(M) := \sum\limits_{i \in I}\frac{|M \cap \Pi_i|}{|M_i \cap \Pi_i|}.\]
If $c(M) \le 1$ for all maximal subrings $M \not\in \mcD$, then $\mcD$ is a minimal cover of the elements of $\Pi$.  Moreover, if $c(M) < 1$ for all maximal subrings $M \not\in \mcD$, then $\mcD$ is the unique minimal cover of the elements of $\Pi$ that uses only maximal subrings.
\end{lem}

\begin{proof}
 The proof is essentially identical to that of \cite[Lemma 5.1]{GaronziKappeSwartz} and \cite[Lemma 3.1]{Swartz} with the word ``group'' changed to the word ``ring,'' but it is included for the sake of completeness.  Let $\mcD$ and $\Pi$ be as in the statement of the lemma, and assume that $c(M) \le 1$ for all maximal subrings not in $\mcD$; the proof with strict inequalities is analogous.  Suppose that $\mcB$ is another cover of the elements of $\Pi$, and let $\mcD' = \mcD \backslash (\mcD \cap \mcB)$ and $\mcB' = \mcB \backslash (\mcD \cap \mcB)$.  The collection $\mcD'$ consists only of subrings from classes $\mathcal{M}_i$, where $i \in I$, and we let $a_i$ be the number of subrings from $\mathcal{M}_i$ in $\mcD'$.  Similarly, the collection $\mcB'$ consists only of subrings from classes $\mathcal{M}_j$, where $j \not\in I$, and we let $b_j$ be the number of subrings from $\mathcal{M}_j$ in $\mcB'$.  Note that since $\mcB$ is a different cover than $\mcD$, we must have $b_j > 0$ for some $j \not\in I$. 

By removing $a_i$ subrings in class $\mcM_i$ from $\mcD$, the new subrings in $\mcB'$ must cover the elements of $\Pi$ that were in these subrings.  Hence, for all $i \in I$,  if $M_k$ denotes a subring in class $\mcM_k$ for each $k$, $$a_i |M_i \cap \Pi_i| \le \sum\limits_{j \not\in I} b_j |M_j \cap \Pi_i|,$$ which in turn implies that, for all $i \in I$, we have $$a_i \le \sum\limits_{j \not\in I} b_j \frac{|M_j \cap \Pi_i|}{|M_i \cap \Pi_i|}.$$  This means that
\begin{align*}
|\mcD'| &= \sum\limits_{i \in I} a_i \le \sum\limits_{i\in I}\sum\limits_{j \not\in I} b_j \frac{|M_j \cap \Pi_i|}{|M_i \cap \Pi_i|} = \sum\limits_{j \not\in I}\sum\limits_{i \in I} b_j \frac{|M_j \cap \Pi_i|}{|M_i \cap \Pi_i|}\\
&= \sum\limits_{j \not\in I}\left(\sum\limits_{i \in I} \frac{|M_j \cap \Pi_i|}{|M_i \cap \Pi_i|}  \right)b_j = \sum\limits_{j \not\in I} c(M_j)b_j \le \sum\limits_{j \not\in I} b_j = |\mcB'|,
\end{align*}
which shows that $$|\mcD| = |\mcD'| + |\mcD \cap \mcB| \le |\mcB'| + |\mcD \cap \mcB| = |\mcB|.$$ Hence, any other cover of the elements of $\Pi$ using only maximal subrings contains at least as many subrings as $\mcD$.  Therefore, $\mcD$ is a minimal cover of $\Pi$.
\end{proof}

While the statement of Lemma \ref{lem:mincovercriterion} can be a bit difficult to parse, the idea is the following: given an appropriate partition of maximal subrings, a set of elements $\Pi$ that is partitioned in a specific way among these maximal subrings, and a potential cover $\mcD$ of the set of elements appropriately derived from the partition of maximal subrings, checking that $\mcD$ is a minimal cover of $\Pi$ amounts to verifying that $c(M) \le 1$ for all maximal subrings not contained in $\mcD$.

We will now construct a set of elements $\Pi$ such that Lemma \ref{lem:mincovercriterion} can be applied to show that $\msC$ is actually a minimal cover.   We start with some terminology from \cite{Crestani}; similar notation and vocabulary can also be found in \cite[Section 7]{LucchiniMarotiARXIV} and \cite{Britnell1}. Let $V$ be an $n$-dimensional vector space over $\F_{q_1}$. Given a positive integer $k$ satisfying $1 \le k < n/2$, we establish a bijection $\phi_k$ from the set $\mcS_k$ of $k$-dimensional subspaces of $V$ to the set $\mcS_{n-k}$ of all $(n-k)$-dimensional spaces of $V$ in such a way that for every $k$-dimensional subspace $U$ of $V$ we have $V = U \oplus \phi_k(U)$.  

For any $k$ and $q$, a Singer cycle in $\GL(k,q)$ is generator of a cyclic subgroup of order $q^k-1$. The basic properties of Singer cycles and the cyclic subgroups they generate that we will use are proved in \cite[pp. 187--189]{Huppert}. All subgroups generated by Singer cycles in $\GL(k,q)$ are conjugate, and the normalizer in $\GL(k,q)$ of a cyclic subgroup generated by a Singer cycle has order $k(q^k-1)$. Hence, the number of subgroups generated by Singer cycles in $\GL(k,q)$ is equal to $|\GL(k,q)|/(k(q^k-1))$.

\begin{Def}\label{def:TypeTk}
Let $1 \le k < n/2$ and let $U$ be a $k$-dimensional subspace of $\F_{q_1}^n$. An element of $M_n(q_1)$ is said to be of \emph{type $T_k$ stabilizing $U$} if---after choosing bases so that $U$ has standard basis $e_1, \ldots, e_k$ and $\phi_k(U)$ has standard basis $e_{k+1}, \ldots, e_n$---it has the form
\begin{equation}\label{formU}
\begin{pmatrix} S_U & 0 \\ 0 & S_{\phi_k(U)} \end{pmatrix},
\end{equation}
where $S_U \in M_k(q_1)$ and $S_{\phi_k(U)} \in M_{n-k}(q_1)$ are Singer cycles. 
\end{Def}

Any matrix having the form of \eqref{formU} lies in the stabilizers of $U$ and $\phi_k(U)$ in $M_n(q_1)$. Moreover, we can count how many such matrices the stabilizer subrings contain. Let $\varphi$ be the Euler totient function. Since the number of choices for $S_U$ and $S_{\phi_k(U)}$ are, respectively,
\begin{equation*}
\dfrac{|\GL(k,q_1)|}{k(q_1^k-1)} \cdot \varphi(q_1^k-1) \; \text{ and } \; \dfrac{|\GL(n-k,q_1)|}{(n-k)(q_1^{n-k}-1)} \cdot \varphi(q_1^{n-k}-1),
\end{equation*}
the stabilizers of $U$ and $\phi_k(U)$ in $M_n(q_1)$ each contain 
\begin{equation}\label{formCount}
\dfrac{|\GL(k,q_1)|}{k(q_1^k-1)} \cdot \dfrac{|\GL(n-k,q_1)|}{(n-k)(q_1^{n-k}-1)} \cdot \varphi(q_1^k-1) \cdot \varphi(q_1^{n-k}-1)
\end{equation}
elements type $T_k$ stabilizing $U$.

Recall our assumption that $q=q_1^d$ and $1 \le d < n-(n/a)$, where $a$ is the smallest prime divisor of $n$. Because of this, one of $d$ or $n-d$ must be less than $n/2$. Assume without loss of generality that $1 \le d < n/2$. 

Given an element $t_U \in M_n(q_1)$ of type $T_d$  stabilizing some $d$-dimensional subspace $U$ of $\F_{q_1}^n$, let $u_0 \in U$ be nonzero, and define $u_i := t_U^i \cdot u_0$ for $1 \le i \le d-1$. Then, $\{u_0, \ldots, u_{d-1} \}$ forms a basis for $U$ because $t_U$ acts regularly on $U\setminus\{0\}$. Let $S_U \in M_d(q_1)$ be the Singer cycle on $U$ corresponding to $t_U$, and let $m(x) \in \F_{q_1}[x]$ be the minimal polynomial of $S_U$. Then, $m(x)$ has degree $d$ and is irreducible over $\F_{q_1}$. For each root $\alpha$ of $m(x)$, we have $\F_{q_1}[\alpha]=\F_q$; note also that $\alpha$ generates $\F_{q_2}$ over $\F_{q_1} \cap \F_{q_2}$. Letting $u := \sum_{i=0}^{d-1} u_i \cdot \alpha^i$, we have $t_U \cdot u = u \cdot \alpha$. Hence, $t_U + \alpha \in C_S(u)$. 

\begin{Def}\label{def:Pi0}
Let $\Pi_0$ be the set of all $t_U + \alpha \in S$, where $U$ runs through the $d$-dimensional subspaces of $\F_{q_1}^n$, $t_U$ has type $T_d$ stabilizing $U$, and $\alpha$ is a root of the minimal polynomial of $S_U$.
\end{Def}

Next, if $d > 1$, then $\omega(d) > 0$ and $\F_{q_2}$ has exactly $\omega(d)$ maximal subfields containing $\F_{q_1} \cap \F_{q_2}$. For each $1 \le i \le \omega(d)$, let $\F_{r_i}$ be the associated maximal subfield of $\F_{q_2}$, and let $d_i = [\F_{r_i}: (\F_{q_1} \cap \F_{q_2})]$. Note that $1 \le d_i < d$. 

\begin{Def}\label{def:Pii}
When $d > 1$, for each $1 \le i \le \omega(d)$ we define $\Pi_i$ to the the set of all $t_W + \gamma \in S$, where $W$ runs through the $d_i$-dimensional subspaces of $\F_{q_1}^n$, $t_W \in M_n(q_1)$ has type $T_{d_i}$ stabilizing $W$, and $\gamma$ is a root of the minimal polynomial of $S_W$.
\end{Def}

Note that $\gamma$ generates $\F_{r_i}$ over $\F_{q_1} \cap \F_{q_2}$. As with $t_U + \alpha$, each element $t_W+\gamma$ centralizes a nonzero $w \in J$. Explicitly, given a nonzero $w_0 \in W$, we can take $w:=\sum_{i=0}^{d_i-1} (t_W^i \cdot w_0) \cdot \gamma^i$, and then $t_W + \gamma \in C_S(w)$.

\begin{Def}\label{def:Pi}
We define
\begin{equation*}
\Pi := \bigcup_{i=0}^{\omega(d)} \Pi_i \subseteq \bigcup_{v \in J \setminus \{0\}} C_S(v).
\end{equation*}
\end{Def}

We next need a partition of the maximal subrings of $S$ so that we may apply Lemma \ref{lem:mincovercriterion}. As noted at the beginning of Section \ref{sect:setup}, the maximal subrings of $S$ are of the form $T_1 \oplus S_2$, where $T_1$ is maximal in $S_1$, or $S_1 \oplus T_2$, where $T_2$ is maximal in $S_2$. Furthermore, the maximal subrings of $S_1$ fall into the three classes listed in Definition \ref{def:maxsubs}.

\begin{Def}\label{def:mcM}
We partition the maximal subrings of $S$ into classes $\mcM_i$, where $i \in I_R$, as follows.
\begin{itemize}
 \item $\mcM_0$ is defined to be the set of all maximal subrings of $S$ of the form $T_U \oplus S_2$, where $T_U$ is a maximal subring of $S_1$ of Type I-$d$.
 \item Assuming $\omega(d) \ge 1$, for each $1 \le i \le \omega(d)$ let $\mcM_i = \{S_1 \oplus T_i\}$, where $T_i $ is the maximal subring of $S_2$ isomorphic to the maximal subfield $\F_{r_i}$ of $\F_{q_2}$ containing $\F_{q_1} \cap \F_{q_2}$.
 \item For each $k \neq d$, $1 \le k \le n - 1$, there is a  set $\mcM_{I,k}$ containing all maximal subrings of $S$ of the form $T \oplus S_2$, where $T$ is a maximal subring of $S_1$ of Type I-$k$.
 \item For each prime $\ell$ dividing $n$, there is a  set $\mcM_{II,\ell}$ containing all maximal subrings of $S$ of the form $T \oplus S_2$, where $T$ is a maximal subring of $S_1$ of Type II-$\ell$.
 \item For each maximal subfield $\F_r$ of $\F_{q_1}$, there is a  set $\mcM_{III,r}$ containing all maximal subrings of $S$ of the form $T \oplus S_2$, where $T$ is a maximal subring of $S_1$ of Type III-$r$.
 \item For each maximal subfield $\F_r$ of $\F_{q_2}$ that does not contain $\F_{q_1} \cap \F_{q_2}$, there is a distinct set $\mcM_{IV,r} := \{S_1 \oplus T\}$, where $T$ is the maximal subfield of $S_2$ isomorphic to $\F_r$.
\end{itemize}
Finally, define $I := \{0, \dots, \omega(d)\} \subseteq I_R.$
\end{Def}

\section{The Covering Number of $A(n,q_1,q_2)$}\label{sect:AGL}

We continue to use the notations defined in Section \ref{sect:setup}. In Theorem \ref{thm:mincoverPi} below, we will assemble all of the disparate pieces constructed in Section \ref{sect:setup} and apply Lemma \ref{lem:mincovercriterion} to conclude that $\msC$ is a minimal cover of $\Pi$. Before doing so, we need some preparatory lemmas. First, we consider which maximal subrings of $M_n(q_1)$ might contain elements of type $T_d$ (cf.\ \cite[Lemma 2.2]{Crestani}, \cite[Lemma 7.4]{LucchiniMarotiARXIV}).

\begin{lem}\label{lem:Tk}
Let $q$ be any prime power, let $k$ be an integer such that $1 \le k < n/2$, let $U$ be a $k$-dimensional subspace of $\F_q^n$, and let $t \in M_n(q)$ be an element of type $T_k$ stabilizing $U$.  If $T$ is a maximal subring of $M_n(q)$ containing $t$, then either
\begin{itemize}
\item $T=T_U$ or $T=T_{\phi_k(U)}$, where $T_U$ and $T_{\phi_k(U)}$ are the Type I-$k$ and Type I-$(n-k)$ subspaces leaving $U$ and $\phi_k(U)$ invariant, respectively; or
\item $T$ is a Type II-$\ell$ maximal subring, where $\ell$ is a common prime divisor of $k$ and $n$.
\end{itemize}
\end{lem}
\begin{proof}
The stated classification of Type I or Type II maximal subrings containing $t$ is shown in \cite[Proposition 7.3]{LucchiniMarotiARXIV} and \cite[Lemma 7.4]{LucchiniMarotiARXIV}. We briefly summarize these arguments for the sake of completeness.

A Singer cycle on an $m$-dimensional vector space over $\F_q$ has minimal polynomial of degree $m$ that is irreducible over $\F_q$. Consequently, the minimal polynomial of $t$ is a product of two irreducible polynomials, one of degree $k$ and the other of degree $n-k$. It follows that if $t$ stabilizes a nonzero proper subspace $W \subseteq \F_q^n$, then either $\dim W = k$ or $\dim W = n-k$, and therefore either $W=U$ or $W=\phi_k(U)$. We conclude that $T_U$ and $T_{\phi_k(U)}$ are the only Type I maximal subrings of $M_n(q)$ containing $t$.

Now, assume that $\ell$ is a common prime divisor of both $n$ and $k$. The cyclic subgroup generated by a Singer cycle in $\GL(k,q)$ is isomorphic to a subgroup of $\GL(k/\ell, q^\ell)$, so $t$ is contained in some Type II-$\ell$ maximal subring of $M_n(q)$. On the other hand, suppose that $\ell$ is a prime divisor of $n$ but $\ell \nmid k$. The centralizer $C$ of $t$ in $\GL(n,q)$ has order $(q^k-1)(q^{n-k}-1)$. If $t$ is contained in a Type II-$\ell$ maximal subring, then $t$ is contained in a copy of $\GL(n/\ell, q^\ell)$, and $C$ contains the group of scalar matrices of this group. Hence, $q^\ell-1$ divides $(q^k-1)(q^{n-k}-1)$. However, since $\ell \nmid k$, $\gcd(q^\ell-1, q^k-1) = q-1 = \gcd(q^\ell-1, q^{n-k}-1)$. It follows that $q^\ell-1$ divides $(q-1)^2$, which is impossible since $\ell \ge 2$.

It remains to show that no Type III maximal subring of $M_n(q)$ contains $t$. Suppose that $\F_r$ is a maximal subfield of $\F_q$ and let $b$ be the prime such that $q=r^b$. Then, the order of $t$ is at least $r^{b(n-k)}-1$. Since $b \ge 2$ and $k < n/2$, we have $b(n-k) \geq 2((n/2) + 1) > n$. As noted in \cite[Lemma 4.5]{Britnell1}, the maximum order of an element of $\GL(n,r)$ is $r^n-1$, which is less than the order of $t$. Thus, $t$ cannot lie in any Type III-$r$ maximal subring of $M_n(q)$.
\end{proof}

Next, we derive some numerical bounds that will be needed when we examine $|M \cap \Pi_i|$ and $|M_i \cap \Pi_i|$.

\begin{lem}\label{lem:nastyestimates}
Let $n \ge 2$, let $1 \le d \le n$, and let $m \ge 1$ be a common divisor of both $n$ and $d$. Then,
\begin{equation*}
\binom{n/m}{d/m}_q \le q^{d/m} \cdot q^{d(n-d)/(m^2)},
\end{equation*}
and
\begin{equation*}
\dfrac{|\GL(d,q)|}{|\GL(d/m, q^m)|} \ge q^{d(d-1)(m-1)/m}.
\end{equation*}
\end{lem}
\begin{proof}
For each $0 \le k \le (d/m)-1$, 
\begin{equation*}
\dfrac{q^{(n/m)-k}-1}{q^{(d/m)-k}-1} \le q^{(n/m)-(d/m)+1},
\end{equation*}
which gives
\begin{equation*}
\binom{n/m}{d/m}_q = \prod_{k=0}^{(d/m)-1} \dfrac{q^{(n/m)-k}-1}{q^{(d/m)-k}-1} \le (q^{(n/m)-(d/m)+1})^{d/m} = q^{d/m} \cdot q^{d(n-d)/(m^2)}.
\end{equation*}
Next, since $|\GL(n,q)| = \prod_{k=0}^{n-1} (q^n-q^k)$, we get
\begin{equation}\label{GLfrac}
\dfrac{|\GL(d,q)|}{|\GL(d/m, q^m)|} = \prod_{k=0, m \nmid k}^{d-1} (q^d-q^k).
\end{equation}
Now, $q^d-q^k \ge q^{d-1}$ for each $k$, and the product in \eqref{GLfrac} has $d-(d/m)$ factors, so
\begin{equation*}
\dfrac{|\GL(d,q)|}{|\GL(d/m, q^m)|} \ge (q^{d-1})^{d-(d/m)} = q^{d(d-1)(m-1)/m},
\end{equation*}
as claimed.
\end{proof}

\begin{lem}\label{lem:moreestimates}
Let $n \ge 5$, let $a$ be the smallest prime divisor of $n$, let $2 \le d < n-(n/a)$, and let $\ell$ be a common prime divisor of both $n$ and $d$. Then,
\begin{equation*}
\binom{n/\ell}{d/\ell}_q \cdot \dfrac{|\GL(d/\ell, q^\ell)|}{|\GL(d,q)|} \cdot \dfrac{|\GL(n/\ell - d/\ell, q^\ell)|}{|\GL(n-d, q)|} \le q^{-(\ell-1)^2-3}.
\end{equation*}
\end{lem}
\begin{proof}
Let 
\begin{equation*}
B := \binom{n/\ell}{d/\ell}_q \cdot \dfrac{|\GL(d/\ell, q^\ell)|}{|\GL(d,q)|} \cdot \dfrac{|\GL(n/\ell - d/\ell, q^\ell)|}{|\GL(n-d, q)|}.
\end{equation*}
Note that $d \ne n/2$, and that the expression for $B$ is symmetric with respect to $d$ and $n-d$. So, we may assume without loss of generality that $d < n/2$. Using Lemma \ref{lem:nastyestimates},
\begin{equation*}
B \le q^{d/\ell} \cdot q^{d(n-d)/(\ell^2)} \cdot q^{-d(d-1)(\ell-1)/\ell} \cdot q^{-(n-d)(n-d-1)(\ell-1)/\ell}.
\end{equation*}
Using the facts that $\frac{d}{\ell} \ge 1$ and $d-1 \ge \ell-1$, we have
\begin{equation}\label{dL1}
\dfrac{d}{\ell} - \dfrac{d(d-1)(\ell-1)}{\ell} = - \dfrac{d}{\ell}((d-1)(\ell-1)-1) \le -(\ell-1)^2+1.
\end{equation}
Next, 
\begin{equation}\label{dL2}
\dfrac{d(n-d)}{\ell^2} - \dfrac{(n-d)(n-d-1)(\ell-1)}{\ell} = \dfrac{n-d}{\ell^2} \Big( d-(n-d-1)(\ell-1)\ell\Big).
\end{equation}
Since $d < n/2$, $n-d-1 \ge n/2$. (Since $\ell > 1$ by assumption, we may exclude $n = 2k + 1$ and $d = k$.)  Hence, 
\begin{equation}\label{dL3}
d-(n-d-1)(\ell-1)\ell \le d-(n/2)(1)(2) = d-n.
\end{equation}
Combining \eqref{dL2} and \eqref{dL3} yields
\begin{equation*}
\dfrac{d(n-d)}{\ell^2} - \dfrac{(n-d)(n-d-1)(\ell-1)}{\ell} \le \dfrac{(n-d)(d-n)}{\ell^2} = -\Big(\dfrac{n-d}{\ell}\Big)^2.
\end{equation*}
Now, $\ell \le d < n/2$, so $n-d > \ell$. Thus, $(n-d)/\ell$ is an integer greater than 1. Consequently, $-((n-d)/\ell)^2 \le -4$. Using this and \eqref{dL1}, we obtain
\begin{equation*}
B \le q^{-(\ell-1)^2+1} \cdot q^{-4} = q^{-(\ell-1)^2-3},
\end{equation*}
which is the desired result.
\end{proof}

\begin{thm}\label{thm:mincoverPi}
 The collection $\msC$ is a minimal cover of $\Pi$.
\end{thm}
\begin{proof}
We first verify that conditions (i)--(iv) of Lemma \ref{lem:mincovercriterion} are satisfied. It is clear that the collections $\mcM_j$ partition the set of all maximal subrings of $S$, and it was shown in Proposition \ref{prop:AGLn>1cover} that $\msC = \bigcup_{i \in I} \mcM_i$ covers $\Pi$. Moreover, $\{ \Pi_i: 0 \le i \le \omega(d)\}$ is a partition of $\Pi$. Indeed, each element of $\Pi$ has the form $t+\beta$, where $t \in S_1$ and $\beta \in S_2$. The element $\beta$ generates $S_2$ if and only if $t+\beta \in \Pi_0$, and $\beta$ generates a maximal subring $\F_{r_i}$ of $S_2$ containing $\F_{q_1} \cap \F_{q_2}$ if and only if $t+\beta \in \Pi_i$. Also, each subring in $\msC$ contains elements of $\Pi$.

Next, consider condition (iv) of Lemma \ref{lem:mincovercriterion}. By construction, each $\Pi_i$ is invariant under conjugation by elements of $S^\times$. As described in \cite[Proposition 3.4]{PeruginelliWerner}, for a fixed $k$ all Type I-$k$ maximal subrings of $M_n(q_1)$ are $\GL(n,q_1)$-conjugates. Analogous conjugacy relations hold for Type II-$\ell$ maximal subrings \cite[Lemma 3.6]{PeruginelliWerner} and Type III-$r$ maximal subrings \cite[Proposition 3.8]{PeruginelliWerner}. Hence, condition (iv) holds whenever $M$ and $M'$ are contained in a class $\mcM_0$, $\mcM_{I,k}$, $\mcM_{II,\ell}$, or $\mcM_{III,r}$. The remaining cases are trivial, because each of the classes $\mcM_{i}$ ($1 \le i \le \omega(d)$) and $\mcM_{IV,r}$ consist of a single maximal subring of $S$.

From here, we must calculate $c(M)$ for $M \notin \msC$. To do this, we must find the cardinalities $|M \cap \Pi_i|$ and $|M_i \cap \Pi_i|$. We begin with $|M_0 \cap \Pi_0|$. Let $M_0 \in \mcM_0$, so $M_0 = T_{U} \oplus S_2$, where $T_{U}$ is a maximal subring of $S_1$ of Type I-$d$ whose elements collectively form the stabilizer in $S_1$ of the $d$-dimensional subspace $U$ of $\F_{q_1}^n$. By Lemma \ref{lem:Tk}, the only elements of $\Pi$ contained in $M_0$ are those in $\Pi_0$, i.e., those of the form $t_{U} + \alpha$, where $t_{U} \in S_1$ is of type $T_d$ stabilizing $U$ and $\alpha \in S_2$ is a root of a particular irreducible polynomial in $\F_{q_1}[x]$ of degree $d$. The number of possibilities for $t_U$ is given in \eqref{formCount}, and there are $d$ possibilities for $\alpha$, so
\begin{equation}\label{eq:M0Pi0}
|M_0 \cap \Pi_0| = \dfrac{|\GL(d,q_1)|}{d(q_1^d-1)} \cdot \dfrac{|\GL(n-d,q_1)|}{(n-d)(q_1^{n-d}-1)} \cdot \varphi(q_1^d-1) \cdot \varphi(q_1^{n-d}-1)\cdot d.
\end{equation}

Next, assuming $\omega(d) > 0$, choose a fixed $i$ such that $1 \le i \le \omega(d)$, which corresponds to a maximal subfield $\F_{r_i}$ of $\F_{q_2}$ containing $\F_{q_1} \cap \F_{q_2}$.  Let $M_i = S_1 \oplus T_{r_i}$, where $T_{r_i}$ is the unique maximal subring of $S_2$ isomorphic to $\F_{r_i}$; note that $\mcM_i = \{M_i\}$. Given $t + \beta \in \Pi$, we have $t + \beta \in M_i$ if and only if the subring of $S_2$ generated by $\beta$ is $T_{r_i}$.  Thus, $M_i \cap \Pi_i = \Pi_i$.  

We count as we did in \eqref{eq:M0Pi0}, except that $M_i$ contains the stabilizer of each $d_i$-dimensional subspace of $\F_{q_1}^n$, of which there are $\binom{n}{d_i}_{q_1}$. This gives us
\begin{align}\label{eq:MiPii}
\begin{split}
|M_i \cap \Pi_i| &= {n \choose {d_i}}_{q_1} \cdot \frac{|\GL(d_i,q_1)|}{d_i(q_1^{d_i}-1)} \cdot \frac{|\GL(n-d_i, q_1)|}{(n-d_i)(q_1^{n-d_i}-1)} \\
& \quad\quad \cdot \varphi(q_1^{d_i} - 1) \cdot \varphi(q_1^{n-d_i} - 1) \cdot d_i.
\end{split}
\end{align}

Now, we will compute $|M \cap \Pi_i|$ for each $0 \le i \le \omega(d)$ and each $M \notin \msC$. If $M \in \mcM_{III,r}$ or $M \in \mcM_{IV,r}$, then, $|M \cap \Pi_i|=0$ for all $0 \le i \le \omega(d)$. Indeed, by Lemma \ref{lem:Tk}, no Type III maximal subring of $S_1$ contains an element of type $T_k$, so $M \cap \Pi_i = \varnothing$ when $M \in \mcM_{III,r}$. On the other hand, if $M \in \mcM_{IV,r}$ for some $r$, then $M = S_1 \oplus \F_r$ for a unique maximal subfield $\F_r$ of $\F_{q_2}$, and $(\F_{q_1} \cap \F_{q_2}) \not\subseteq \F_r$. However, for any $t+\beta \in \Pi$, $\beta$ generates a subfield of $\F_{q_2}$ containing $\F_{q_1} \cap \F_{q_2}$.

Next, assume $M \in \mcM_{I,k}$, where $k \ne d$. By Lemma \ref{lem:Tk}, $M \cap \Pi_{i'} \neq \varnothing$ for a unique value $i' \in I$. Suppose first that $i'=0$. Then, $M$ stabilizes an $(n-d)$-dimensional subspace $W$ of $\F_{q_1}^n$. Since these are in one-to-one correspondence with $d$-dimensional subspaces via the bijection $\phi_d$, $M$ contains all elements of type $T_d$ stabilizing $\phi_d^{-1}(W)$. Let $M_0 \in \mcM_0$ be the maximal subring in $\mcM_0$ that stabilizes $\phi_d^{-1}(W)$. Then, $M \cap \Pi_0 = M_0 \cap \Pi_0$, so 
\begin{equation*}
c(M) =  \frac{|M \cap \Pi_0|}{|M_0 \cap \Pi_0|} = 1.
\end{equation*}
On the other hand, if $i' > 0$ then $k=d_i$ or $k=n-d_i$. In either case, $\mcM_{i'}$ consists of a single subring $M_{i'}$, and $M_{i'} \cap \Pi_{i'} = \Pi_{i'}$. So,
\begin{equation*}
c(M) = \frac{|M \cap \Pi_{i'}|}{|M_{i'} \cap \Pi_{i'}|} = \frac{|M \cap \Pi_{i'}|}{|\Pi_{i'}|} \le 1.
\end{equation*}

The only remaining possibility is that $M \in \mcM_{II, \ell}$ for some prime divisor $\ell$ of $n$. Assume that $M = T \oplus S_2 \cong M_{n/\ell}(q_1^\ell) \oplus \F_{q_2}$.  Since each $d_i$ divides $d$, by Lemma \ref{lem:Tk}, $M \cap \Pi = \varnothing$ if $\ell$ does not divide $\gcd(n, d)$.  For the case when $\ell$ divides both $n$ and $d$, it is possible that $M$ contains elements from some or all sets $\Pi_i$. Note that for this to occur, $n \ge 5$, and $d$ (or $d_i$, if $\ell \mid d_i$) is at least 2. Counting as in \eqref{formCount} and \eqref{eq:M0Pi0}, we obtain
\begin{equation*}
|M \cap \Pi_0| \le \dfrac{|\GL(d/\ell,q_1^\ell)|}{\tfrac{d}{\ell}(q_1^d-1)} \cdot \dfrac{|\GL((n-d)/\ell,q_1^\ell)|}{\tfrac{n-d}{\ell}(q_1^{n-d}-1)} \cdot \varphi(q_1^d-1) \cdot \varphi(q_1^{n-d}-1)\cdot d.
\end{equation*}
Using this, \eqref{eq:M0Pi0}, and Lemma \ref{lem:moreestimates} yields
\begin{align}\label{eq:MPi0}
\begin{split}
\frac{|M \cap \Pi_0|}{|M_0 \cap \Pi_0|} &\le \ell^2 \cdot {n/\ell \choose d/\ell}_{q_1} \cdot \frac{|\GL(d/\ell, q_1^\ell)|}{|\GL(d, q_1)|} \cdot \frac{|\GL(n/\ell - d/\ell, q_1^\ell)|}{|\GL(n-d, q_1)|}\\
&\le \dfrac{\ell^2}{q_1^{(\ell-1)^2+3}}.
\end{split}
\end{align}
Similarly, using \eqref{eq:MiPii} in place of \eqref{eq:M0Pi0}, for $1 \le i \le \omega(d)$, we get
\begin{equation}\label{eq:MPii}
\frac{|M \cap \Pi_i|}{|M_i \cap \Pi_i|} \le \binom{n}{d_i}_{q_1}^{-1} \cdot \dfrac{\ell^2}{q_1^{(\ell-1)^2+3}} \le \binom{n}{1}_{q_1}^{-1} \cdot \dfrac{\ell^2}{q_1^{(\ell-1)^2+3}} \le \dfrac{\ell^2}{q_1^{n+(\ell-1)^2+3}}.
\end{equation}
Certainly, $\omega(d) \le q_1^n$, so 
\begin{equation*}
c(M) = \sum_{i = 0}^{\omega(d)} \dfrac{|M \cap \Pi_i|}{|M_i \cap \Pi_i|} \le \dfrac{\ell^2}{q_1^{(\ell-1)^2+3}} + \dfrac{\omega(d)\ell^2}{q_1^{n+(\ell-1)^2+3}} \le \dfrac{2\ell^2}{q_1^{(\ell-1)^2+3}},
\end{equation*}
and this last fraction is less than 1 because both $\ell$ and $q_1$ are at least 2.

We have shown that $c(M) \le 1$ for any $M \notin \msC$. By Lemma \ref{lem:mincovercriterion}, $\msC$ is a minimal cover of the elements of $\Pi$.
\end{proof}

\begin{cor}\label{cor:AGLn>1}
Let $n \ge 3$, let $a$ be the smallest prime divisor of $n$, and let $q = q_1 \otimes q_2 = q_1^d$ with $d < n-(n/a)$. There exists a cover of $R$ of size 
 \[ q^n + {n \choose d}_{q_1} + \omega(d),\]
 and, if $R$ is $\sigma$-elementary, then 
 \[ \sigma(R) = q^n + {n \choose d}_{q_1} + \omega(d).\]
\end{cor}

\begin{proof}
Applying Propositions \ref{prop:AGLcover}(1) and \ref{prop:AGLn>1cover} and recalling that $|\msS(R)|=q^n$ gives
\begin{equation*}
\sigma(R) \le |\msS(R)| + |\msZ| = q^n + {n \choose d}_{q_1} + \omega(d).
\end{equation*}
If $R$ is $\sigma$-elementary, then by Proposition \ref{prop:AGLcover}(2) and Corollary \ref{cor:AGLn>1}, $\msS(R) \cup \msZ$ is a minimal cover of $R$, so $\sigma(R)=|\msS(R)| + |\msZ|$.
\end{proof}

It remains to determine when $R$ is $\sigma$-elementary. This can be done by comparing the upper bound on $\sigma(R)$ in Corollary \ref{cor:AGLn>1} with $\sigma(M_n(q_1))$.

\begin{lem}\label{lem:prodlowerbound}
Let $n \ge 2$, let $a$ be the smallest prime divisor of $n$, and let $q$ be any prime power. Then, 
\begin{equation*}
\dfrac{1}{a}\prod_{k=1, a \nmid k}^{n-1} (q^n - q^k) \ge q^{n(n- (n/a)-1)},
\end{equation*}
with equality only when $(n,q)=(2,2)$ or $(3,2)$.
\end{lem}
\begin{proof}
Let $P := (1/a)\prod_{k=1, a \nmid k}^{n-1} (q^n - q^k)$. We will verify the result directly for $2 \le n \le 4$. When $n=2$, we have $P=\tfrac{1}{2}(q^2-q) \ge 1$, with equality only when $q=2$. Similarly, when $n=3$, $P=\tfrac{1}{3}(q^3-q)(q^3-q^2) \ge q^3$, with equality only when $q=2$. Next, when $n=4$, $P=\tfrac{1}{2}(q^4-q)(q^4-q^3) > q^4$.

From here, assume that $n \ge 5$. Let $\delta = n - (n/a)$, which is the number of factors in the product in $P$. Since $a \nmid (n-1)$, $q^n - q^{n-1}$ always occurs as a factor in $P$, and the other $\delta-1$ factor are bounded below by $q^n - q^{n-2}$. So,
\begin{align*}
P &\ge \dfrac{1}{a} (q^n - q^{n-2})^{\delta - 1} (q^n - q^{n-1})\\
&= \dfrac{1}{a} \Big[ q^{n-1} \cdot \Big(\dfrac{q^2-1}{q}\Big) \Big]^{\delta-1}(q^{n-1})(q-1)\\
&= \dfrac{1}{a}(q^{n-1})^\delta (q-1) \Big(\dfrac{q^2-1}{q}\Big)^{\delta-1}.
\end{align*}

Now, $(n-1)\delta - n(\delta-1) = n - \delta = n/a$. So, to prove that $P > q^{n(\delta-1)}$, it will be enough to show that
\begin{equation}\label{eq:sufficient}
\dfrac{1}{a}(q^{n/a})(q-1) \Big(\dfrac{q^2-1}{q}\Big)^{\delta-1} > 1.
\end{equation}
Suppose first that $a = n$. Then, $\delta = n-1$, and the left-hand side of \eqref{eq:sufficient} is bounded below by
\begin{equation*}\label{eq:n=ell}
\dfrac{1}{n}(2)(1)\Big(\dfrac{3}{2}\Big)^{n-2}.
\end{equation*}
This quantity is greater than 1 since $n \ge 5$, so the result holds when $a = n$. So, assume that $a < n$. Since $n \ne 4$, we have $2 \le a \le (n/3)$. This gives us the following lower bounds for the terms in \eqref{eq:sufficient}: $(1/a) \ge (3/n)$, $q \ge 2$, $(n/a) \ge 3$, and $\delta \ge (n/2)$. Applying these bounds to the left-hand side of \eqref{eq:sufficient} yields
\begin{equation*}
\dfrac{3}{n}(2^3)(1)\Big(\dfrac{3}{2}\Big)^{(n/2)-1},
\end{equation*}
which is greater than 1 for $n \ge 5$, as desired.
\end{proof}

\begin{thm}\label{thm:AGLupperbound}
Let $n \ge 3$, let $a$ be the smallest prime divisor of $n$, and let $q = q_1 \otimes q_2 = q_1^d$. If $d < n - (n/a)$, then $\sigma(R) \le \sigma(M_n(q_1))$, with equality only when $n=3$ and $q_1=2$. Thus, $R$ is $\sigma$-elementary if $(n,q_1) \ne (3,2)$.
\end{thm}
\begin{proof}
Recall that
\begin{equation*}
\sigma(M_n (q_1)) = \frac{1}{a} \prod_{k=1,\\ a \nmid k}^{n-1} (q_1^n - q_1^k) + \sum_{k=1,\\ a \nmid k}^{\lfloor n/2 \rfloor} {n \choose k}_{q_1}.
\end{equation*}
By Corollary \ref{cor:AGLn>1}, $\sigma(R)$ is bounded above by
\begin{equation*}
B:= q^n + \binom{n}{d}_{q_1} + \omega(d)
\end{equation*}
Hence, it suffices to show that $B < \sigma(M_n(q_1))$ except when $n=3$ and $q_1=2$. Assume that $d < n - (n/a)$. To ease the notation, let $f = \lfloor n/2 \rfloor$.

The cases $n=3$ and $n=4$ (both of which force $d=1$) can be verified by inspection. Note that if $n=3$ and $q=2$, then $B=15=\sigma(M_3(2))$. So, assume that $n \ge 5$. By Lemma \ref{lem:prodlowerbound},
\begin{equation*}
q^n = q_1^{nd} \le q_1^{n(n-(n/a)-1)} < \frac{1}{a}\prod_{k=1, a \nmid k}^{n-1} (q_1^n-q_1^k).
\end{equation*}
Next, the sum
\begin{equation}\label{eq:thesum}
\sum_{k=1, a \nmid k}^f \binom{n}{k}_{q_1}
\end{equation}
has at least two summands, one of which is $\binom{n}{1}_{q_1}$. Certainly, $\omega(d) < n < \binom{n}{1}_{q_1}$. If $n \not\equiv 0$ mod $4$, then $\binom{n}{f}_{q_1}$ occurs as a summand of \eqref{eq:thesum}, and $\binom{n}{d}_{q_1} \le \binom{n}{f}_{q_1}$. If $n \equiv 0$ mod $4$, then $d \le f-1$, and $\binom{n}{f-1}_{q_1}$ occurs in \eqref{eq:thesum}, so $\binom{n}{d}_{q_1} \le \binom{n}{f-1}_{q_1}$. Thus, in all cases,
\begin{equation*}
q^n + \binom{n}{d}_{q_1} + \omega(d) \le \frac{1}{a}\prod_{k=1, a \nmid k}^{n-1} (q_1^n-q_1^k) + \sum_{k=1, a \nmid k}^f \binom{n}{k}_{q_1},
\end{equation*}
as desired.
\end{proof}

Finally, we are able to prove our main result.

\begin{proof}[Proof of Theorem \ref{thm:main}]
The cases where $n \le 2$ or $d \ge n-(n/a)$ follow from Theorem \ref{thm:smalln}. When $n \ge 3$, $d < n-(n/a)$, and $R$ is $\sigma$-elementary, the value of $\sigma(R)$ is given in Corollary \ref{cor:AGLn>1}, which itself follows from Theorem \ref{thm:mincoverPi} and the work done in Section \ref{sect:setup}. Finally, the classification of $\sigma$-elementary $R$ with $n \ge 3$ was completed in Theorem \ref{thm:AGLupperbound}.
\end{proof}

\bibliographystyle{plain}
\bibliography{CoveringNumbersOfRings}

\end{document}